\documentclass[10pt,reqno]{amsart}
\usepackage{amsmath,graphicx,amssymb,xcolor, ulem, cancel}
\usepackage{datetime}
\usepackage[hidelinks]{hyperref}

\makeatletter
\DeclareRobustCommand{\gaussk}{\DOTSB\gaussk@\slimits@}
\newcommand{\gaussk@}{\mathop{\vphantom{\sum}\mathpalette\bigcal@{K}}}

\newcommand{\bigcal@}[2]{%
  \vcenter{\m@th
    \sbox\z@{$#1\sum$}%
    \dimen@=\dimexpr\ht\z@+\dp\z@
    \hbox{\resizebox{!}{0.8\dimen@}{$\mathcal{K}$}}%
  }%
}
\newcommand{\cfracplus}{\mathbin{\cfracplus@}}

\newcommand{\cfracplus@}{%
  \sbox\z@{$\dfrac{1}{1}$}%
  \sbox\tw@{$+$}%
  \raisebox{\dimexpr\dp\tw@-\dp\z@\relax}{$+$}%
}
\newcommand{\cfracdots}{\mathord{\cfracdots@}}
\newcommand{\cfracdots@}{%
  \sbox\z@{$\dfrac{1}{1}$}%
  \sbox\tw@{$+$}%
  \raisebox{\dimexpr\dp\tw@-\dp\z@\relax}{$\cdots$}%
}
\makeatother

\newtheorem{theorem}{Theorem}
\newtheorem{corollary}{Corollary}

\newtheorem{lemma}{Lemma}

\theoremstyle{definition}

\theoremstyle{remark}
\newtheorem{remark}{Remark}
\numberwithin{equation}{section}
\numberwithin{theorem}{section}
\numberwithin{corollary}{section}
\numberwithin{remark}{section}
\numberwithin{lemma}{section}

\usepackage{enumerate}
\definecolor{airforceblue}{rgb}{0.36, 0.54, 0.66}
\definecolor{azure}{rgb}{0.0, 0.5, 1.0}
\definecolor{bluegray}{rgb}{0.4, 0.6, 0.8}
\definecolor{coolblack}{rgb}{0.0, 0.18, 0.39}
 \definecolor{darkmidnightblue}{rgb}{0.0, 0.2, 0.4}
 \definecolor{indigo}{rgb}{0.0, 0.25, 0.42}
  \definecolor{mediumelectricblue}{rgb}{0.01, 0.31, 0.59}
  \definecolor{royalazure}{rgb}{0.0, 0.22, 0.66}
  \definecolor{royalbluetraditional}{rgb}{0.0, 0.14, 0.4}
  \definecolor{ultramarineblue}{rgb}{0.25, 0.4, 0.96}
  \definecolor{unitednationsblue}{rgb}{0.36, 0.57, 0.9}

\definecolor{babyblueeyes}{rgb}{0.63, 0.79, 0.95}
 \definecolor{cornflowerblue}{rgb}{0.39, 0.58, 0.93}
 \definecolor{skyblue}{rgb}{0.53, 0.81, 0.92}
 \definecolor{yaleblue}{rgb}{0.06, 0.3, 0.57}

 \definecolor{darktangerine}{rgb}{1.0, 0.66, 0.07}
  \definecolor{orangepeel}{rgb}{1.0, 0.62, 0.0}

 \definecolor{deepchestnut}{rgb}{0.73, 0.31, 0.28}
 \definecolor{firebrick}{rgb}{0.7, 0.13, 0.13}
 \definecolor{mordantred19}{rgb}{0.68, 0.05, 0.0}
 \definecolor{richmaroon}{rgb}{0.69, 0.19, 0.38}

 \definecolor{dimgray}{rgb}{0.41, 0.41, 0.41}
 \definecolor{warmblack}{rgb}{0.0, 0.26, 0.26}

  \definecolor{forestgreen}{rgb}{0.0, 0.27, 0.13}
  \definecolor{gray-asparagus}{rgb}{0.27, 0.35, 0.27}
  \definecolor{huntergreen}{rgb}{0.21, 0.37, 0.23}
  \definecolor{yellow-green}{rgb}{0.6, 0.8, 0.2}

\definecolor{harvardcrimson}{rgb}{0.79, 0.0, 0.09}

\definecolor{fluorescentpink}{rgb}{1.0, 0.08, 0.58}
\definecolor{richmaroon}{rgb}{0.69, 0.19, 0.38}
\definecolor{darkgreen}{rgb}{0.0, 0.2, 0.13}
\definecolor{pumpkin}{rgb}{1.0, 0.46, 0.09}
\definecolor{gold}{rgb}{0.83, 0.69, 0.22}
\definecolor{goldenrod}{rgb}{0.85, 0.65, 0.13}
\definecolor{mikadoyellow}{rgb}{1.0, 0.77, 0.05}
\definecolor{selectiveyellow}{rgb}{1.0, 0.73, 0.0}
\definecolor{yellowmunsell}{rgb}{0.94, 0.8, 0.0}
\definecolor{cobalt}{rgb}{0.0, 0.28, 0.67}
\definecolor{sapphire}{rgb}{0.03, 0.15, 0.4}
\definecolor{coolblack}{rgb}{0.0, 0.18, 0.39}
\definecolor{royalblue(traditional)}{rgb}{0.0, 0.14, 0.4}
\definecolor{tyrianpurple}{rgb}{0.4, 0.01, 0.24}
\definecolor{forestgreen(traditional)}{rgb}{0.0, 0.27, 0.13}

 \definecolor{purpleheart}{rgb}{0.41, 0.21, 0.61}
 \definecolor{regalia}{rgb}{0.32, 0.18, 0.5}
 \definecolor{apricot}{rgb}{0.98, 0.81, 0.69}

\allowdisplaybreaks

\renewcommand{\pmod}[1]{\;\allowbreak(\operatorname{mod}#1)}

\begin{document}
	\title[On Schultz's generalization of Borweins' cubic identity]{On Schultz's generalization of Borweins' cubic identity}

\author{Heng Huat Chan}
\address{Mathematical Research Center, Shandong University, No. 1 Building, 5 Hongjialou Road, Jinan, P.R.C. China 250100.}
\email{chanhh6789@sdu.edu.cn}

\author{Song Heng Chan}
\address{Division of Mathematical Sciences, School of Physical and Mathematical Sciences, Nanyang Technological University, 21 Nanyang link, Singapore, 637371, Singapore}

\email{chansh@ntu.edu.sg}

\author{Zhi-Guo Liu}
\address{School of Mathematical Sciences, Key Laboratory of MEA (Ministry of Education) Shanghai Key Laboratory of PMMP, East China Normal University, Shanghai 200241, P.R.China}
\email{zgliu@math.ecnu.edu.cn}

\author{Wadim Zudilin}

\address{Department of Mathematics, IMAPP, Radboud University, PO Box 9010, 6500~GL Nijmegen, Netherlands}
\email{w.zudilin@math.ru.nl}

\date{\today}

\begin{abstract}
In 1991, the Borweins established a cubic analogue of Jacobi's identity for theta functions, which is used by B.C.~Berndt, S.~Bhargava, and F.G.~Garvan
in the development of Ramanujan's
cubic theory of elliptic functions. In 2013,
D.~Schultz discovered an identity for theta series in three variables which generalizes the Borweins' identity.
In this article, we revisit Schultz's identity and present two distinct approaches to its derivation. Our investigation not only provides
new proofs but also yields several new Schultz-type identities.
\end{abstract}

\maketitle

\maketitle

\section{Introduction}
\label{s1}

Around 1991, J.M.~Borwein and P.B.~Borwein \cite[(2.3)]{Borweins}
published the identity
\begin{align}
\left(\sum_{m,n\in \mathbf Z} \omega^{m-n}q^{m^2+mn+n^2}\right)^3&+\left(\sum_{m,n\in \mathbf Z}  q^{(m+1/3)^2+(m+1/3)(n+1/3)+(n+1/3)^2}\right)^3\notag\\
 &\qquad\qquad= \left(\sum_{m,n\in \mathbf Z} q^{m^2+mn+n^2}\right)^3,\label{Borweins}
\end{align}
where $\omega=e^{2\pi i /3}$.
Their identity is a cubic analogue of Jacobi's classical identity \cite[p.~90, (14)]{Jacobi},
\cite[Theorem 3.7, Corollary 3.8]{Cooper} given by
\begin{equation}
\left(\sum_{m,n\in \mathbf Z} (-1)^{m+n}q^{m^2+n^2}\right)^2
+\left(\sum_{m,n\in \mathbf Z} q^{(m+1/2)^2+(n+1/2)^2}\right)^2=\left(\sum_{m,n\in \mathbf Z} q^{m^2+n^2}\right)^2.
\label{Jacobi}
\end{equation}
Let
$$
{}_2F_1\left(a,b;c;z\right)
=\sum_{j=0}^\infty \dfrac{(a)_j(b)_j}{(c)_j}\dfrac{z^j}{j!},$$
where
$$(a)_j= \begin{cases} \qquad1\quad &\text{when $j=0$,} \\
  \displaystyle{\prod_{k=1}^j (a+k-1) } \quad &\text{when $j\in \mathbf Z^+$.}
\end{cases}$$
The Borweins used \eqref{Borweins} to show that
\begin{equation}\label{Borweins-inversion}
\sum_{m,n\in \mathbf Z} q^{m^2+mn+n^2}
={}_2F_1\left(\dfrac{1}{3},\dfrac{2}{3};1;\kappa(q)\right)
\end{equation}
where
$$\kappa(q)=
\dfrac{c^3(q)}{a^3(q)}$$
with
$$c(q) = \sum_{m,n\in \mathbf Z} q^{(m+1/3)^2+(m+1/3)(n+1/3)+(n+1/3)^2}$$
and
$$a(q)=\sum_{m,n\in \mathbf Z} q^{m^2+mn+n^2}.$$
Borweins' inversion formula \eqref{Borweins-inversion} is crucial, for example, in proving Ramanujan's identity  \cite[(31)]{Ramanujan-pi}, \cite{Chan-Liaw}
\[
\dfrac{27}{4\pi}=
\sum_{j=0}^\infty (15j+2)\dfrac{(1/3)_j(2/3)_j(1/2)_j}{(1)_j^3} \left(\dfrac{2}{27}\right)^j.
\]
Using Borweins' identities \eqref{Borweins} and \eqref{Borweins-inversion}, B.C.~Berndt, S.~Bhargava and F.G.~Garvan \cite{BBG} developed Ramanujan's theory of elliptic functions to the cubic base. An alternative approach to Ramanujan's theory of elliptic functions to the cubic base without the use of \eqref{Borweins} was recently discovered by H.H.~Chan and Z.G.~Liu \cite{Chan-Liu-RIM}.

In 2013, D.~Schultz \cite{Schultz} published a two-variable generalization of~\eqref{Borweins}.

\begin{theorem}
\label{th-Schultz}
The following identity is valid:
\begin{align}
&\left(\sum_{m,n\in \mathbf Z} \omega^{m-n} q^{(m+1/3)^2+(m+1/3)(n+1/3)+(n+1/3)^2} x^{m+1/3}y^{n+1/3}\right)^3 \notag\\
&\quad
+ \left(\sum_{m,n\in \mathbf Z} q^{m^2+mn+n^2}x^my^n\right)^3
=\left(\sum_{m,n\in \mathbf Z} \omega^{m-n}  q^{m^2+mn+n^2}x^my^n\right)^3
\notag \\ &\qquad\qquad +\left(\sum_{m,n\in \mathbf Z}
q^{(m+1/3)^2+(m+1/3)(n+1/3)+(n+1/3)^2}x^{m+1/3}y^{n+1/3}\right)^3.
\label{Borweins-xy}
\end{align}
\end{theorem}
We shall refer to \eqref{Borweins-xy} as Schultz's identity.
To see that \eqref{Borweins} is a specialization 
of \eqref{Borweins-xy}, we need to set $x=y=1$ in \eqref{Borweins-xy}
and show that
\begin{equation}\label{Borweins-0}
\sum_{m,n\in \mathbf Z} \omega^{m-n} q^{(m+1/3)^2+(m+1/3)(n+1/3)+(n+1/3)^2} =0.
\end{equation}
Identity \eqref{Borweins-0}
is equivalent to $S=0$ with
\[
S=\sum_{m,n\in \mathbf Z} \omega^{m-n} q^{m^2+mn+n^2+m+n}.
\]
Replacing $(m,n)$ by $(n,-m-n-1)$ in the latter double sum, we deduce that
\[
S=\omega^{-1}\cdot S,
\]
which implies that $S=0$ and completes the proof of~\eqref{Borweins-0}.
For other proofs of \eqref{Borweins-0}, see
\cite[Corollary]{XP} and
\cite[Section 2.3.3]{Schultz-thesis}.

The first proof of Theorem~\ref{th-Schultz} was provided by Schultz \cite[p.~622]{Schultz}.
On page 663 of \cite{Schultz}, Schultz mentioned four possible methods to establishing Theorem~\ref{th-Schultz}.
His proof was carried out using  his first
method. Schultz's remaining three methods indicated that Theorem~\ref{th-Schultz} may be proved by
\begin{enumerate}[(1)]
\item multiplying out and combining the series on both sides of \eqref{Borweins-xy},
\item verifying that both sides of \eqref{Borweins-xy} agree when written in terms of some basis of the theta functions of order 3,
\item verifying that both sides of \eqref{Borweins-xy} agree up to a certain finite order at a finite set of points.
\end{enumerate}
Schultz mentioned the difficulties of proving \eqref{Borweins-xy} using the above methods but gave little details.
Around 2017, X.R.~Ma and R.Z.~Wei \cite[Corollary 4.12]{Ma-Wei} gave a new proof of Theorem~\ref{th-Schultz} which did not follow any of Schultz's suggestions.

In this article, we give two new proofs of Theorem~\ref{th-Schultz} which are also independent of Schultz's insights.
Our first approach is motivated by the work of X.M.~Yang \cite{Yang}. In Section~\ref{s2}, we show that \eqref{Borweins-xy} follows from one of her identities \cite[(3.8)]{Yang}.
In Section~\ref{s4}, we use transformation formulas for theta functions and the identities in Sections~\ref{s2} and \ref{s3} to derive analogues of \eqref{Borweins-xy}.
Our second proof of Theorem~\ref{th-Schultz} connects \eqref{Borweins-xy} to the study of Macdonald's identities \cite{Chan-Cooper-Toh}, and this is presented in Section~\ref{s4}.

Motivated by Theorem~\ref{th-Schultz}, we found an analogue of \eqref{Borweins-xy} for Jacobi's identity \eqref{Jacobi}.

\begin{theorem}
\label{th-Jacobi}
We have the following identity:
\begin{align}
&\left(\sum_{m,n\in \mathbf Z} (-1)^{m+n} q^{(m+1/2)^2+(n+1/2)^2}x^{m+1/2}y^{n+1/2}\right)^2+\left(\sum_{m,n\in \mathbf Z} q^{m^2+n^2}x^my^n\right)^2
\notag\\
 &=\left(\sum_{m,n\in \mathbf Z}
q^{(m+1/2)^2+(n+1/2)^2}x^{m+1/2}y^{n+1/2}\right)^2+\left(\sum_{m,n\in \mathbf Z} (-1)^{m+n} q^{m^2+n^2}x^my^n\right)^2.
\label{Jacobi-xy}
\end{align}
\end{theorem}

Identity \eqref{Jacobi-xy} is not new and, presumably, was known to Jacobi; its traces can
be found in classical textbooks, e.g.,  as part of Exercise~1 in \cite[Chap.~21]{Watson-Whittaker} or \cite[p.~9]{Lawden}.

Since
\begin{align*}
\sum_{j=-\infty}^\infty (-1)^j q^{(j+1/2)^2}
&=\sum_{\ell=-\infty}^\infty (-1)^{\ell-1}q^{(\ell-1/2)^2}\\
&=\sum_{k=-\infty}^\infty (-1)^{-k-1} q^{(-k-1/2)^2}\\
&=-\sum_{k=-\infty}^\infty (-1)^k q^{(k+1/2)^2},
\end{align*}
we conclude that
\begin{equation}\label{Jacobi-0}
\sum_{m,n\in \mathbf Z} (-1)^{m+n} q^{(m+1/2)^2+(n+1/2)^2} =0.
\end{equation}
By \eqref{Jacobi-0}, we see that \eqref{Jacobi} follows from Theorem~\ref{th-Jacobi} with the substitution  $x=y=\nobreak 1$.

In Section~\ref{s5}, we present and prove several new extensions of Theorem \ref{th-Jacobi}.
In Section \ref{s6}, we derive further identities equivalent to \eqref{Borweins}.
The identities presented in these last two sections are motivated by the recent discoveries of H.H.~Chan, S.H.~Chan and P.~Sole \cite{Chan-Chan-Sole}.

\section{Jacobi's theta functions and the first proof of Schultz's identity}
\label{s2}

Throughout this work, $q=e^{\pi i\tau}$ where the imaginary part of
$\tau$ is positive. The Jacobi theta functions are defined by
\begin{align*}
&\vartheta_1(z|\tau)=-i\sum_{n=-\infty}^{\infty}(-1)^n q^{(2n+1)^2/4} e^{(2n+1)iz},\\
&\vartheta_2(z|\tau) =\sum_{n=-\infty}^\infty q^{(2n+1)^2/4} e^{(2n+1)iz},\\
&\vartheta_3(z|\tau) =\sum_{n=-\infty}^\infty q^{n^2} e^{2niz}
\intertext{and}
&\vartheta_4(z|\tau) =\sum_{n=-\infty}^\infty (-1)^nq^{n^2} e^{2niz}.
\end{align*}

A straightforward calculation also yields the following expression for $\vartheta_1$:
\begin{equation}\label{T1-sum}
\vartheta_1(z|\tau)=2 \sum_{n=0}^\infty (-1)^n q^{(2n+1)^2/4}\sin
(2n+1)z.
\end{equation}
The infinite product representations of $\vartheta_1(z|\tau)$ and $\vartheta_3(z|\tau)$ (see \cite[p.~469]{Watson-Whittaker}, \cite[p.~30]{Chan-book}), which are equivalent to Jacobi's triple product identity, are given by
\begin{align}
\vartheta_1(z|\tau)&=2q^{1/4}(\sin z)\prod_{n=1}^{\infty} (1-q^{2n}) (1-q^{2n} e^{2iz}) (1-q^{2n} e^{-2iz}) \label{JTP-T1}\intertext{and}
\vartheta_3(z|\tau)&=\prod_{n=1}^{\infty} (1-q^{2n}) (1+q^{2n-1} e^{2iz}) (1+q^{2n-1} e^{-2iz}).\notag
\end{align}
The Jacobi theta function $\vartheta_1(z|\tau)$ satisfies
\begin{equation}\label{eqn1}
\vartheta_1(z|\tau)=-\vartheta_1(z+\pi|\tau)=-qe^{2iz}\vartheta_1(z+\pi\tau|\tau)
\end{equation}
and
\begin{equation}\label{eqn2}
\vartheta_1(z+(\pi+\pi\tau)/2|\tau)=q^{-1/4} e^{-iz} \vartheta_3(z|\tau).
\end{equation}
Let
\begin{equation*} 
\Lambda=\left\{x\pi+y\pi\tau\bigm|0\le x< 1,\, 0\le y < 1 \right\}
\end{equation*}
be the fundamental period parallelogram.
In $\Lambda$,
$\vartheta_1(z|\tau)$  has a simple zero at $z=0$  and $\vartheta_3(z|\tau)$ has a simple
zero at $z=(\pi+\pi\tau)/2$.
Using \eqref{eqn1} and \eqref{eqn2}, it is immediate that
$$\vartheta_1(3z|3\tau),\quad
   e^{2iz}\vartheta_1(3z+\pi\tau|3\tau) \quad\text{and}\quad
   e^{-2iz}\vartheta_1(3z-\pi\tau|3\tau)$$
satisfy the relations
\begin{equation}\label{quasi-rel}
f(z)=-f(z+\pi)=-q^{3} e^{6iz} f(z+\pi\tau).
\end{equation}
The next lemma plays an important role in our approach to \eqref{Borweins-xy}.

\begin{lemma} \label{independence} The functions $\vartheta_1(3z|3\tau)$,
   $e^{2iz}\vartheta_1(3z+\pi\tau|3\tau)$ and
   $e^{-2iz}\vartheta_1(3z-\pi\tau|3\tau)$ are linearly independent over $\mathbf C$.
\end{lemma}
\begin{proof}
Suppose there are complex numbers $c_1, c_2$ and $c_3$ such that
\begin{equation}\label{ind:eqn1}
c_1\vartheta_1(3z|3\tau)+c_2e^{2iz}\vartheta_1(3z+\pi\tau|3\tau)+c_3e^{-2iz}\vartheta_1(3z-\pi\tau|3\tau)=0.
\end{equation}
If we set $z=0$ in \eqref{ind:eqn1} and use the fact that $\vartheta_1(0|3\tau)=0$,
we find that
\begin{equation*}
(c_2-c_3)\vartheta_1(\pi\tau|3\tau)=0.
\end{equation*}
Since $\vartheta_1(\pi\tau|3\tau)=iq^{-1/6}\eta(\tau)\not=0$, we deduce that $c_2=c_3$.
Substituting $c_2=c_3$ into \eqref{ind:eqn1}, we obtain
\begin{equation}\label{ind:eqn2}
c_1\vartheta_1(3z|3\tau)+c_2\left(e^{2iz}\vartheta_1(3z+\pi\tau|3\tau)+e^{-2iz}\vartheta_1(3z-\pi\tau|3\tau)\right)=0.
\end{equation}
Next, we let $z=\pi/3$ in \eqref{ind:eqn2}. By \eqref{T1-sum}, $\vartheta_1(\pi|3\tau)=0$ and therefore,
\begin{equation*}
-i\sqrt{3}\vartheta_1(\pi\tau|3\tau)c_2=0,
\end{equation*}
which implies that $c_2=0$ since $\vartheta_1(\pi\tau|3\tau)\neq 0$.
Thus, $c_2=c_3=0$.
Identity \eqref{ind:eqn1} reduces to
$$c_1\vartheta_1(3z|3\tau)=0$$
which implies that $c_1=0$ since
$\vartheta_1(3z|3\tau)\neq 0$. The three functions are therefore linearly independent over $\mathbf {C}$.
\end{proof}

\begin{theorem}
\label{thm1}
Let $z,x,y\in \mathbf C$ and $q=e^{\pi i \tau}$ with $\operatorname{Im}\tau>0$. Take
\begin{equation}\label{A}
A(x,y|\tau) = \sum_{\substack{k,\ell\in \mathbf Z}} q^{2(\ell^2+\ell k+k^2)}e^{2ix(2\ell+k)}e^{2iy(2k+\ell)}
\end{equation}
and
\begin{equation}
C(x,y|\tau) = e^{2ix}e^{2iy}\sum_{\substack{k,\ell\in \mathbf Z}} q^{2(\ell^2+\ell k+k^2+k+\ell+1/3)}e^{2ix(2\ell+k)}e^{2iy(2k+\ell)}.\label{C}
\end{equation}
Then
\begin{align}\label{cubthetaen3}
&\vartheta_1(z+x|\tau)\vartheta_1(z+y|\tau)\vartheta_1(z-x-y|\tau)
=-A(x, y|\tau)\vartheta_1(3z|3\tau)\\
&\quad +q^{1/3} C(-x, -y|\tau) e^{2iz}\vartheta_1(3z+\pi\tau|3\tau)
+q^{1/3} C(x, y|\tau) e^{-2iz}\vartheta_1(3z-\pi\tau|3\tau)\nonumber
\end{align}
and
\begin{align}\label{cubthetaen4}
&\vartheta_3(z+x|\tau)\vartheta_3(z+y|\tau)\vartheta_3(z-x-y|\tau)
=A(x, y|\tau)\vartheta_3(3z|3\tau)\\
&\quad +q^{1/3} C(-x, -y|\tau) e^{2iz}\vartheta_3(3z+\pi\tau|3\tau)
+q^{1/3} C(x, y|\tau) e^{-2iz}\vartheta_3(3z-\pi\tau|3\tau).\nonumber
\end{align}
\end{theorem}

\begin{proof}
Introduce $$T(z,x,y|\tau)=\vartheta_1(z+x|\tau)\vartheta_1(z+y|\tau)\vartheta_1(z-x-y|\tau).$$
This function is constructed so that it satisfies \eqref{quasi-rel},
since
$$T(z,x,y|\tau)= -q^{3}e^{2i(z+x+z+y+z-x-y)} T(z+\pi\tau,x,y|\tau)=-q^3e^{6iz}T(z+\pi\tau,x,y|\tau).$$
Therefore, the functions
\begin{align*}
\frac{\vartheta_1(3z|3\tau)}{T(z,x,y|\tau)}, \quad
\frac{e^{2iz}\vartheta_1(3z+\pi\tau|3\tau)}{T(z,x,y|\tau)} \quad\text{and}\quad
\frac{e^{-2iz}\vartheta_1(3z-\pi\tau|3\tau)}{T(z,x,y|\tau)}
\end{align*}
are all elliptic functions with
periods $\pi$ and $\pi\tau$ and with 3~poles in $\Lambda$.
Since an elliptic function with a simple pole must be a constant, we conclude that there exists
$c, \lambda_1, \lambda_2$ and $\lambda_3$ which are  dependent on $x$ and $y$ but independent of $z$  such that
$$cT(z,x,y|\tau)
=\lambda_1\vartheta_1(3z|3\tau)
+\lambda_2 e^{2iz}\vartheta_1(3z+\pi\tau|3\tau)+\lambda_3 e^{-2iz}\vartheta_1(3z-\pi\tau|3\tau).$$
Note that by Lemma \ref{independence}, $c\neq 0$.
By replacing $\lambda_j$ by $\lambda_j/c$, we find that
\begin{align}
&
\vartheta_1(z+x|\tau)\vartheta_1(z+y|\tau)\vartheta_1(z-x-y|\tau)
\notag\\ &\qquad
=\lambda_1\vartheta_1(3z|3\tau) +\lambda_2  e^{2iz}\vartheta_1(3z+\pi\tau|3\tau)
+\lambda_3  e^{-2iz}\vartheta_1(3z-\pi\tau|3\tau). \label{cubthetaen5}
\end{align}
Replacing $z$ by $z+(\pi+\pi\tau)/2$ and using \eqref{eqn2}, we may rewrite \eqref{cubthetaen5} as
\begin{align}
&\vartheta_3(z+x|\tau)\vartheta_3(z+y|\tau)\vartheta_3(z-x-y|\tau) \notag\\
&\qquad=-\lambda_1\vartheta_3(3z|3\tau) +\lambda_2  e^{2iz}\vartheta_3(3z+\pi\tau|3\tau)
+\lambda_3  e^{-2iz}\vartheta_3(3z-\pi\tau|3\tau). \label{cubthetaen5-5}
\end{align}
Identity \eqref{cubthetaen5} written in the form \eqref{cubthetaen5-5} allows us to determine $\lambda_1,\lambda_2$ and $\lambda_3$.
Substituting the series representation of $\vartheta_3(z|\tau)$ into \eqref{cubthetaen5-5} and equating the constant terms, the coefficients
of $e^{2iz}$ and $e^{-2iz}$, we deduce that
\begin{align*}
\lambda_1&=-\sum_{\substack{k,\ell, m\in \mathbf Z\\ k+\ell+m=0}} q^{\ell^2+k^2+m^2}e^{2i((\ell-m)x+(k-m)y)}
\\ & =
-\sum_{\substack{k,\ell\in \mathbf Z}} q^{2(\ell^2+\ell k+k^2)}e^{2i(2\ell+k)x}e^{2i(2k+\ell)y},\\
\lambda_2&= \sum_{\substack{k,\ell, m\in \mathbf Z\\ k+\ell+m=1}} q^{\ell^2+k^2+m^2}e^{2i((\ell-m)x+(k-m)y)} \\
&= q^{1/3}e^{-2ix}e^{-2iy}\sum_{\substack{k,\ell\in \mathbf Z}} q^{2(\ell^2+\ell k+k^2-k-\ell+1/3)}e^{2i(2\ell+k)x}e^{2i(2k+\ell)y}\intertext{and}
\lambda_3&=\sum_{\substack{k,\ell, m\in \mathbf Z\\ k+\ell+m=-1}} q^{\ell^2+k^2+m^2}e^{2i((\ell-m)x+(k-m)y)} \\
&=q^{1/3}e^{2ix}e^{2iy}\sum_{\substack{k,\ell\in \mathbf Z}} q^{2(\ell^2+\ell k+k^2+k+\ell+1/3)}e^{2i(2\ell+k)x}e^{2i(2k+\ell)y}.
\end{align*}
Thus, the required identities \eqref{cubthetaen3} and \eqref{cubthetaen4} follow from \eqref{cubthetaen5} and \eqref{cubthetaen5-5}.
\end{proof}

We are now ready to prove Theorem \ref{th-Schultz}.

\begin{proof}[First proof of Theorem~\ref{th-Schultz}]
First, let $z=-\pi \tau/3$ in \eqref{cubthetaen3} to deduce that
\begin{align*}
&\vartheta_1(-\pi\tau/3+x|\tau)\vartheta_1(-\pi\tau/3+y|\tau)\vartheta_1(-\pi\tau/3-x-y|\tau)
\
\\ &\qquad\qquad
\qquad=-A(x, y|\tau)\vartheta_1(-\pi\tau|3\tau)
+q C(x, y|\tau) \vartheta_1(-2\pi\tau|3\tau).\nonumber
\end{align*}
By \eqref{JTP-T1}, we know that
\begin{equation}\label{cubthetaen7}
\vartheta_1(-\pi\tau|3\tau)=-i q^{-1/3} \eta(\tau)\quad\text{and}\quad \vartheta_1(-2\pi\tau|3\tau)=-i q^{-4/3} \eta(\tau),
\end{equation}
where $$\eta(\tau)=q^{1/12}\prod_{k=1}^\infty (1-q^{2k}).$$
Therefore,
$$-\dfrac{iq^{1/3}}{\eta(\tau)}\vartheta_1(-\pi\tau/3+x|\tau)
\vartheta_1(-\pi\tau/3+y|\tau)\vartheta_1(-\pi\tau/3-x-y|\tau)
=A(x, y|\tau)
-C(x, y|\tau),$$
which, after replacing $\tau$ by $3\tau$, implies that
\begin{equation*}
A(x, y|3\tau)
-C(x, y|3\tau)=-\dfrac{iq}{\eta(3\tau)}\vartheta_1(-\pi\tau+x|3\tau)
\vartheta_1(-\pi\tau+y|3\tau)\vartheta_1(-\pi\tau-x-y|3\tau).\end{equation*}
Using \eqref{JTP-T1}, we deduce that
\begin{align}&
A(x,y|3\tau)-C(x,y|3\tau)=\prod_{k=1}^\infty (1-q^{6k})^2(1-q^{6k-2}e^{2ix}) (1-q^{6k-4}e^{-2ix})\notag\\&\qquad\times
\prod_{k=1}^\infty (1-q^{6k-2}e^{2iy})(1-q^{6k-4}e^{-2iy})
(1-q^{6k-2}e^{2i(-x-y)})
(1-q^{6k-4}e^{2i(x+y)}).\label{key-to-Schultz}
\end{align}
Next, note that
$$C(x,y|3\tau)=q^2 C^*(x,y|3\tau)$$ where
$C^*(x,y|3\tau)$ is a function of $q^6$.
Therefore,
\begin{align}&A^3(x,y|3\tau)-C^3(x,y|3\tau)=A^3(x,y|3\tau)-q^6 {C^*}^3(x,y|3\tau) \notag \\
&\qquad = \prod_{j=0}^2 \left(A(x,y|3\tau)-e^{2\pi i j/3}q^2 C^*(x,y|3\tau)\right) \notag \\
&\qquad =\prod_{k=1}^\infty (1-q^{6k})^6
(1-q^{18k-6}e^{6ix})(1-q^{18k-12}e^{-6ix})(1-q^{18k-6}e^{6iy}) \notag \\
&\qquad\times \prod_{k=1}^\infty
(1-q^{18k-12}e^{-6iy})(1-q^{18k-6}e^{6i(-x-y)})
(1-q^{18k-12}e^{6i(x+y)}).\label{product-cube}
\end{align}
But observe that
$A^3(x+\pi/3,y|3\tau)-C^3(x+\pi/3,y|3\tau)$ has the same product representation as $A^3(x,y|3\tau)-C^3(x,y|\tau)$. Therefore,
$$A^3(x,y|3\tau)+C^3(x+\pi/3,y|3\tau)
=A^3(x+\pi/3,y|3\tau)+C^3(x,y|3\tau)$$ or
\begin{equation}\label{Borweins-xy-1}
A^3(x,y|\tau)+C^3(x+\pi/3,y|\tau)
=A^3(x+\pi/3,y|\tau)+C^3(x,y|\tau).
\end{equation}
Identity \eqref{Borweins-xy-1} is equivalent to Schultz's identity \eqref{Borweins-xy}.
\end{proof}

\begin{remark}
The proof of \eqref{product-cube} using \eqref{key-to-Schultz} is inspired by the proof of Theorem~2.3 in \cite{Borweins-Garvan}.
The final step in the above proof of Theorem~\ref{th-Schultz}, using the fact that the product on the right hand side of \eqref{product-cube} can be written as
a difference of the cube of two theta series in two ways, is inspired by the proof of Corollary~4.2 in \cite{Ma-Wei}.
\end{remark}

\section{Analogues of Schultz's identity}
\label{s3}

In this section, we derive several identities analogous to \eqref{Borweins-xy}.

\begin{theorem}
\label{th3.1}
Let $C(x,y|\tau)$ be given by \eqref{C}. Then
\begin{equation}\label{Borweins-xy-variant-1}
C^3(x, y|\tau)-C^3(-x, -y|\tau)=C^3(x+\pi/3, y|\tau)-C^3(-x-\pi/3, -y|\tau).
\end{equation}
\end{theorem}

\begin{proof}
By setting $z=0$ in \eqref{cubthetaen3} and simplifying using \eqref{cubthetaen7}, we deduce that
\begin{equation}\label{cubthetaen8}
C(x, y|\tau)-C(-x, -y|\tau)=\frac{-i}{\eta(\tau)} \vartheta_1(x|\tau)\vartheta_1(y|\tau)\vartheta_1(x+y|\tau),
\end{equation}
since $\vartheta_1(0|3\tau)=0$.
Next, set $z=\pi/3$ in \eqref{cubthetaen3}. Note that $\vartheta_1(\pi|3\tau)=0$ and therefore, simplifying
using \eqref{cubthetaen7}, we deduce that
\begin{equation}\label{cubthetaen9}
C(x, y|\tau)-\omega^2 C(-x, -y|\tau)=\frac{-i\omega}{\eta(\tau)} \vartheta_1\left(\frac{\pi}{3}+x|\tau\right)\vartheta_1\left(\frac{\pi}{3}+y|\tau\right)\vartheta_1\left(\frac{\pi}{3}-x-y|\tau\right)
\end{equation}
where $\omega=\exp (2\pi i/3)$. From \eqref{cubthetaen8} and \eqref{cubthetaen9}, we find that
\begin{align}\label{cubthetaen10}
C(x, y|\tau)&=\frac{-\omega}{\sqrt{3}\eta(\tau)} \vartheta_1(x|\tau)\vartheta_1(y|\tau)\vartheta_1(x+y|\tau)\\
&\quad +\frac{1}{\sqrt{3}\eta(\tau)} \vartheta_1\left(\frac{\pi}{3}+x|\tau\right)\vartheta_1\left(\frac{\pi}{3}+y|\tau\right)\vartheta_1\left(\frac{\pi}{3}-x-y|\tau\right).\nonumber
\end{align}
Replacing $(x, y)$ by $(-x, -y)$ in \eqref{cubthetaen9} and simplifying, we find that
\begin{equation}\label{cubthetaen11}
C(x, y|\tau)-\omega C(-x, -y|\tau)=\frac{i\omega^2}{\eta(\tau)} \vartheta_1\left(\frac{\pi}{3}-x|\tau\right)\vartheta_1\left(\frac{\pi}{3}-y|\tau\right)\vartheta_1\left(\frac{\pi}{3}+x+y|\tau\right).
\end{equation}
Using \eqref{JTP-T1}, we deduce that
\begin{equation}\label{cubthetaen12}
\vartheta_1(z|\tau)\vartheta_1\left(\frac{\pi}{3}+z|\tau\right)\vartheta_1\left(\frac{\pi}{3}-z|\tau\right)
=\frac{\eta^3(\tau)}{\eta(3\tau)} \vartheta_1(3z|3\tau).
\end{equation}
Multiplying \eqref{cubthetaen8}, \eqref{cubthetaen9} and \eqref{cubthetaen11} and using
\eqref{cubthetaen12}, we immediately find that
\begin{equation}\label{cubthetaen13}
C^3(x, y|\tau)-C^3(-x, -y|\tau)=-i\frac{\eta^6(\tau)}{\eta^3(3\tau)}
\vartheta_1(3x|3\tau)\vartheta_1(3y|3\tau)\vartheta_1(3x+3y|3\tau).
\end{equation}
The right-hand side of \eqref{cubthetaen13} is invariant under the change $x\to x+\pi/3$, and this leads to the required identity \eqref{Borweins-xy-variant-1}, which is a variant of \eqref{Borweins-xy-1}.
\end{proof}

\begin{corollary}
\label{th3.2}
We have the identity
\begin{equation}\label{cubthetaen15}
\frac{\eta^9(\tau)}{\eta^3(3\tau)}
=\left(\sum_{m, n=-\infty}^{\infty} q^{2(m^2+mn+n^2)}\right)^3
-\left(\sum_{m, n=-\infty}^{\infty} q^{2(m^2+mn+n^2+m+n+1/3)}\right)^3.
\end{equation}	
\end{corollary}

\begin{proof}
Let $x=y=\pi\tau/3$ in \eqref{cubthetaen13}. Using \eqref{cubthetaen7}, we find that
\begin{equation}\label{cubthetaen14}
q^2C^3(-\pi\tau/3, -\pi\tau/3|\tau)-q^2C^3(\pi\tau/3, \pi\tau/3|\tau)
=\frac{\eta^9(\tau)}{\eta^3(3\tau)}.
\end{equation}	
Using the definition $C(x, y|\tau)$ given by \eqref{C} and simplifying, we deduce that \eqref{cubthetaen14} is equivalent to~\eqref{cubthetaen15}.
\end{proof}

\begin{corollary}
\label{th3.3}
The following identity is valid:
\begin{equation}\label{b(q)-prod}
\frac{\eta^3(\tau)}{\eta(3\tau)}=\sum_{m, n=-\infty}^{\infty} \omega^{m-n} q^{2(m^2+mn+n^2)}.
\end{equation}
\end{corollary}

\begin{proof}
Replacing $z$ by $z+(\pi+\pi\tau)/2$ in \eqref{cubthetaen12}, we find that
\begin{equation}\label{cubthetaen16}
\vartheta_3(z|\tau)\vartheta_3\left(\frac{\pi}{3}+z|\tau\right)\vartheta_3\left(\frac{\pi}{3}-z|\tau\right)
=\frac{\eta^3(\tau)}{\eta(3\tau)} \vartheta_3(3z|3\tau).
\end{equation}
Equating the constant terms on both sides, we obtain identity \eqref{b(q)-prod}.
\end{proof}

Combining \eqref{b(q)-prod} and \eqref{cubthetaen15}, we arrive at \eqref{Borweins}.
This also shows that \eqref{Borweins-xy-variant-1} may be viewed as
a new generalization of \eqref{Borweins}.

\begin{remark}
For an elegant proof of Theorem~\ref{th3.3}, see \cite[Proposition 2.2(i)]{Borweins-Garvan}.
\end{remark}

\begin{corollary}
\label{th3.4}
The following is valid:
\begin{equation}
\frac{\eta^3(3\tau)}{\eta(\tau)} = \frac13\sum_{m, n=-\infty}^{\infty}q^{2(m^2+mn+n^2+m+n+1/3)}.
\label{id-th3.4}
\end{equation}
\end{corollary}

\begin{proof}
Setting $x=y$ in \eqref{cubthetaen8}, we find that
\begin{equation}\label{cubthetaen18}
\sum_{m, n=-\infty}^{\infty}q^{2(m^2+mn+n^2+m+n+1/3)} \sin(2(3m+3n+2)x)=\frac{-1}{2\eta(\tau)} \vartheta_1^2(x|\tau)\vartheta_1(2x|\tau).
\end{equation}
Letting $x=\pi/3$ in \eqref{cubthetaen18} and using $\vartheta_1(\pi/3|\tau)=\sqrt{3}\eta(3\tau)$, we arrive at \eqref{id-th3.4}.
\end{proof}

\begin{corollary}
\label{th3.5}
We have the identity
\begin{align}
-3\sqrt{3} i \prod_{n=1}^{\infty} \frac{(1-q^{2n})^6 (1-q^{18n})^3}{(1-q^{6n})^3}
&=\omega^2\left(\sum_{m, n=-\infty}^{\infty} q^{2(m^2+mn+n^2+m+n)} \omega^{m+n}\right)^3
\nonumber\\ &\kern-6mm
-\omega^{-2}\left(\sum_{m, n=-\infty}^{\infty} q^{2(m^2+mn+n^2+m+n)} \omega^{-m-n}\right)^3.\label{newcubthetan1}
\end{align}
\end{corollary}

\begin{proof}
This follows from setting $x=y=\pi/9$ in \eqref{cubthetaen13}  and once again using
$\vartheta_1(\pi/3|\tau)=\sqrt{3}\eta(3\tau)$.
\end{proof}

Identity \eqref{newcubthetan1} is of the form
\begin{equation}\label{cubic-theta-pythagoras} a X^3+bY^3=cZ^3\end{equation}
where $X,Y,Z$ are weight one modular forms and $a,b,c\in \mathbf C$.
Note that Borweins' identity \eqref{Borweins} also satisfies \eqref{cubic-theta-pythagoras}
with $a=b=c=1$.

\medskip
In the remaining of this section, we continue our search for identities similar to \eqref{Borweins-xy}.
Recall from \cite[(5), p.~205]{Schoeneberg} that if $M$ is a real symmetric $f\times f$ matrix, $\mathbf m\in \mathbf Z^f$ written as a column vector,
$\mathbf x$ is a column vector with $f$ complex variable components $x_1,x_2,\cdots, x_f$, then for $\operatorname{Im}\tau>0$, we have
\begin{equation}\label{Sc-T}
\sum_{\mathbf m\in \mathbf Z^f} e^{\pi i \tau(\mathbf m+\mathbf x)^tM(\mathbf m+\mathbf x)}
=\dfrac{1}{(\sqrt{-i\tau})^f|\text{det}(M)|^{1/2}}
\sum_{\mathbf m\in \mathbf Z^f} e^{-\pi i (\mathbf m^t M^{-1} \mathbf m)/\tau+2\pi i\mathbf m^t\mathbf x}.
\end{equation}
Let $f=2$, $M=\begin{pmatrix} 2 & 1\\ 1 & 2\end{pmatrix}$, $\mathbf m=(m,n)^t$ and $\mathbf x=(x/(\pi\tau),y/(\pi\tau))^t$ in \eqref{Sc-T}. We deduce immediately that if
\begin{equation}\label{Adagger}
A^{\dagger}(x,y|\tau) =
\sum_{m,n\in \mathbf Z} q^{2(m^2-mn+n^2)}e^{2imx}e^{2iny},
\end{equation}
then
$$A(x,y|\tau) =\dfrac{1}{-i\tau\sqrt{3}}\exp\left(-\dfrac{2i(x^2+xy+y^2)}{\pi\tau}\right)A^\dagger\left(\dfrac{x}{\tau},\dfrac{y}{\tau}\bigg|-\dfrac{1}{3\tau}\right),$$
which is equivalent to
\begin{equation}\label{trans-A} A\left(\dfrac{x}{3\tau},\dfrac{y}{3\tau}\bigg|-\dfrac{1}{3\tau}\right) =-i\sqrt{3}\tau \exp\left(\dfrac{2i(x^2+xy+y^2)}{3\pi\tau}\right)A^\dagger\left(-x,-y|\tau\right).
\end{equation}
Observe that if we replace $m$ by $m+n$ in \eqref{Adagger}, then
\begin{equation}\label{Adagger-2} A^{\dagger}(x,y|\tau)=
\sum_{m,n\in \mathbf Z} q^{2(m^2+mn+n^2)}e^{2i(m+n)x}e^{2iny}.
\end{equation}
From now on, we will use \eqref{Adagger-2} instead of \eqref{Adagger}.
Similarly, if
\begin{equation}\label{Cdagger}
C^{\dagger}(x,y|\tau) = \sum_{m,n\in \mathbf Z} \omega^{m+n} q^{2(m^2-mn+n^2)}e^{2imx}e^{2iny},
\end{equation}
then
\begin{equation}\label{trans-C} C\left(\dfrac{x}{3\tau},\dfrac{y}{3\tau}\bigg|-\dfrac{1}{3\tau}\right) =-i\sqrt{3}\tau \exp\left(\dfrac{2i(x^2+xy+y^2)}{3\pi\tau}\right)C^\dagger\left(-x,-y|\tau\right).
\end{equation}
Observe that if we replace $m$ by $m+n$ in \eqref{Adagger}, then
$$C^{\dagger}(x,y|\tau)= \sum_{m,n\in \mathbf Z} \omega^{m-n} q^{2(m^2+mn+n^2)}e^{2i(m+n)x}e^{2iny}.$$

\begin{theorem}\label{thm2} The following identities hold:
\begin{align*}
&3\vartheta_1(z+x|\tau)\vartheta_1(z+y|\tau)\vartheta_1(z-x-y|\tau)=A^{\dagger}\left(-x,-y\Big|\frac{\tau}{3}\right)\vartheta_1\left(z\Big| \frac{\tau}{3}\right)
\notag  \\
&\qquad\qquad -C^{\dagger}\left(x, y\Big|\frac{\tau}{3}\right)\vartheta_1\left(z-\frac{\pi}{3}\Big| \frac{\tau}{3}\right)
-C^{\dagger}\left(-x, -y\Big|\frac{\tau}{3}\right)\vartheta_1\left(z+\frac{\pi}{3}\Big| \frac{\tau}{3}\right),
\end{align*}
\begin{align*}
&3\vartheta_3(z+x|\tau)\vartheta_3(z+y|\tau)\vartheta_3(z-x-y|\tau)=A^{\dagger}\left(-x, -y\Big|\frac{\tau}{3}\right)\vartheta_3\left(z\Big| \frac{\tau}{3}\right)
\notag  \\
&\qquad\qquad +C^{\dagger}\left(x, y\Big|\frac{\tau}{3}\right)\vartheta_3\left(z-\frac{\pi}{3}\Big| \frac{\tau}{3}\right)
+C^{\dagger}\left(-x, -y\Big|\frac{\tau}{3}\right)\vartheta_3\left(z+\frac{\pi}{3}\Big| \frac{\tau}{3}\right),
\end{align*}
\begin{align*}
&\left(C^\dagger(x,y|\tau)\right)^3-
\left(C^\dagger(-x,-y|\tau)\right)^3 \\
&\qquad \qquad =q^2e^{4ix+2iy}\left(C^\dagger(x+\pi\tau,y|\tau)\right)^3-
q^2e^{4ix+2iy}\left(C^\dagger(-x-\pi\tau,-y|\tau)\right)^3
\intertext{and}
&\left(A^\dagger(x,y|\tau)\right)^3-
\left(C^\dagger(x,y|\tau)\right)^3 \\
&\qquad \qquad =q^2e^{4ix+2iy}\left(A^\dagger(x+\pi\tau,y|\tau)\right)^3-
q^2e^{4ix+2iy}\left(C^\dagger(x+\pi\tau,y|\tau)\right)^3.
\end{align*}
\end{theorem}

\begin{proof}
Using \eqref{trans-A}, \eqref{trans-C} and the transformation formula \cite[p.~475]{Watson-Whittaker}, \cite[p.~35]{Siegel}
\begin{align}\label{imeqn1}
\vartheta_j\left(\frac{z}{\tau}\Big|-\frac{1}{\tau}\right)=\sqrt{-i\tau} \exp\left(\frac{iz^2}{\pi \tau}\right) \vartheta_j(z|\tau),
\end{align}
for $j=1,3$, we deduce the identities claimed from Theorem~\ref{thm1}.
\end{proof}

\section{Second proof of Schultz's identity}
\label{s4}

We will now give a second proof of \eqref{key-to-Schultz} which leads to a second proof  of Theorem \ref{th-Schultz}.

\begin{proof}[Second proof of Theorem~\ref{th-Schultz}]
First, we observe that
$$\{(m,n)\mid m,n\in \mathbf Z\}
=\{(k+\ell,k-\ell)\mid k,\ell\in \mathbf Z\}\cup \{(k+\ell+1,k-\ell)\mid k,\ell\in \mathbf Z\}.$$
Let $\gamma=\exp(\pi i/3)=(1+\sqrt{-3})/2$. Note that $\gamma$ is a sixth-root of unity and $\gamma^2=\omega$. Let
$$N(m+n\gamma) = (m+n\gamma)(m+n\overline{\gamma})=m^2+mn+n^2.$$
Therefore,
\begin{align*}\sum_{m,n\in \mathbf Z} &q^{m^2+mn+n^2}e^{iu(m+n)}e^{iv(m-n)} \\
&=\sum_{m,n\in \mathbf Z} q^{N(m+n\gamma)}e^{iu(m+n)}e^{iv(m-n)} \\
&=\sum_{k,\ell\in \mathbf Z} q^{N((k+\ell)+(k-\ell)\gamma)}e^{2kiu}e^{2\ell iv}
+\sum_{k,\ell\in \mathbf Z} q^{N((k+\ell+1)+(k-\ell)\gamma)}e^{iu(2k+1)}e^{iv(2\ell+1)}
 \\
&=\sum_{k,\ell\in \mathbf Z} q^{3k^2+\ell^2}e^{2iku}e^{2i\ell v}
+\sum_{k,\ell\in \mathbf Z} q^{3(k+1/2)^2+(\ell+1/2)^2}e^{iu(2k+1)}e^{iv(2\ell +1)} \\
&=\vartheta_3(u|3\tau)\vartheta_3(v|\tau)+\vartheta_2(u|3\tau)\vartheta_2(v|\tau)
\end{align*}
or
\begin{equation}\label{A-theta}\sum_{m,n\in \mathbf Z} q^{2(m^2+mn+n^2)}e^{2iu(m+n)}e^{2iv(m-n)}
=\vartheta_3(2u|6\tau)\vartheta_3(2v|2\tau)+\vartheta_2(2u|6\tau)\vartheta_2(2v|2\tau).
\end{equation}
 Note that by the transformation $u=3(x+y)/2$ and $v=(x-y)/2$ in \eqref{A-theta} and using \eqref{A}, we deduce that
\begin{equation}\label{A-theta-100}
A(x,y|\tau)
=\vartheta_3(3x+3y|6\tau)\vartheta_3(x-y|2\tau)+\vartheta_2(3x+3y|6\tau)\vartheta_2(x-y|2\tau).
\end{equation}
Identity \eqref{A-theta-100} shows that in order to prove \eqref{key-to-Schultz}, it suffices to show that
\begin{align}
&\vartheta_3(6u|6\tau)\vartheta_3(2v|2\tau)+\vartheta_2(6u|6\tau)\vartheta_2(2v|2\tau)\notag
\\
&\qquad - q^{2/3}e^{4iu}\left(
\vartheta_3(6u+2\pi\tau|6\tau)\vartheta_3(2v|2\tau)+
\vartheta_2(6u+2\pi\tau|6\tau)\vartheta_2(2v|2\tau)\right)\notag \\
&\quad=-e^{4iu}q^{2/3}\dfrac{i}{\eta(\tau)}\vartheta_1(2u-2\pi\tau/3|\tau)\vartheta_1(u+v-\pi\tau/3|\tau)\vartheta_1(u-v-\pi\tau/3|\tau).
\label{CCT-Y}
\end{align}
We will now show that \eqref{CCT-Y} follows from an identity found by Chan, Cooper and Toh \cite[Theorem 3.1]{Chan-Cooper-Toh} given by
\begin{equation}\label{CCT}
G_2(2x|2\tau)\vartheta_2(2y|2\tau)+G_3(2x|2\tau)\vartheta_3(2y|2\tau)
=\dfrac{1}{\eta(\tau)}\vartheta_1(2x|\tau)\vartheta_1(x+y|\tau)
\vartheta_1(x-y|\tau),
\end{equation}
where
$$G_2(2z|2\tau) = 2\sum_{\alpha\equiv 1 \pmod{6}}\!\!\!\!\!\! q^{\alpha^2/6}
\sin{2\alpha z}$$
and
$$G_3(2z|2\tau) = 2\sum_{\alpha\equiv -2 \pmod{6}} \!\!\!\!\!\! q^{\alpha^2/6}
\sin{2\alpha z}.$$

Using $\sin{u}=(e^{iu}-e^{-iu})/(2i)$, we find that
$$G_2(2u|2\tau) = \dfrac{1}{i}\left(G_{2,1}(2u|2\tau)-G_{2,2}(2u|2\tau)\right)$$
where
$$G_{2,1}(2u|2\tau) = \sum_{k=-\infty}^\infty q^{6k^2+2k+1/6}e^{2iu(6k+1)}$$
and
$$G_{2,2}(2u|2\tau) = \sum_{k=-\infty}^\infty q^{6k^2+2k+1/6}e^{-2iu(6k+1)}.$$
Next, replace $u$  by $u-\pi\tau/3$.
One can verify that
$$G_{2,1}(2u-2\pi\tau/3|2\tau) = e^{8iu}\vartheta_2(6u+2\pi\tau|6\tau)$$
and
$$G_{2,2}(2u-2\pi\tau/3|2\tau) = q^{-2/3}e^{4iu}\vartheta_2(6u|6\tau).$$
Similarly, if we write
$$G_3(2u|2\tau) = \dfrac{1}{i}\left(G_{3,1}(2u|2\tau)-G_{3,2}(2u|2\tau)\right)$$
where
$$G_{3,1}(2u|2\tau) = \sum_{k=-\infty}^\infty q^{6k^2-4k+2/3}e^{2iu(6k-2)}$$
and
$$G_{3,2}(2u|2\tau) = \sum_{k=-\infty}^\infty q^{6k^2-4k+2/3}e^{-2iu(6k-2)},$$
we find that
$$G_{3,1}(2u-2\pi\tau/3|2\tau) = e^{8iu}\vartheta_3(6u+2\pi\tau|6\tau)$$
and
$$G_{3,2}(2u-2\pi\tau/3|2\tau) = q^{-2/3}e^{4iu}\vartheta_3(6u|6\tau).$$
Substituting the above into \eqref{CCT}, we arrive at \eqref{CCT-Y}.
\end{proof}

\section{Generalizations of Schult'z identity}
\label{s5}

Recently, Chan,  Chan and Sole discover the identities  \cite[(1.5), (1.6)]{Chan-Chan-Sole},
\begin{align}
&\left(\sum_{m,n=-\infty}^\infty (-1)^{m+n} q^{dm^2+bmn+dn^2}\right)^2
+ \left(2\sum_{m,n=-\infty}^\infty q^{2(d(m+1/2)^2+b(m+1/2)n+dn^2)}\right)^2\notag\\
&\qquad\qquad\qquad\qquad =\left(\sum_{m,n=-\infty}^\infty q^{dm^2+bmn+dn^2}
\right)^2\label{ChanChanSole-id}
\end{align}
and
\begin{align}&\left(\sum_{m,n=-\infty}^\infty (-1)^{m+n} q^{bm^2+bmn+dn^2}\right)^2
+\left(\sum_{m,n=-\infty}^\infty q^{2(b(m+1/2)^2+b(m+1/2)n+dn^2)}\right)^2\notag\\
&\qquad\qquad\qquad\qquad=\left(\sum_{m,n=-\infty}^\infty q^{2(bm^2+bmn+dn^2)}\right)^2, \label{CCS-gen}
 \end{align}
 where $b$ and $d$ satisfy certain conditions.
It is natural to ask if there are any two-variable extensions for \eqref{ChanChanSole-id} and \eqref{CCS-gen} similar to the generalization for
\eqref{Jacobi} and \eqref{Borweins}.
The answer is affirmative.

\begin{theorem}\label{AC2thm}
Let
\begin{align} \label{A2def}
&A_{d,b}(u,v|\tau)=\sum_{m,n\in\mathbf{Z}} q^{d m^2+b mn+d n^2}e^{2iu(m+n)} e^{2iv(m-n)}
\intertext{and} \nonumber
&C_{d,b}(u,v|\tau)\\  \label{C2def}
&=\sum_{m,n\in\mathbf{Z}} q^{d m^2+ b mn+d n^2+(d +b/2)m+(d+b/2)n+d/2+b/4}e^{2iu(m+n+1)} e^{2iv(m-n)}.
\end{align}
Let $b$ and $d$ be integers satisfying $-2d < b< 2d$ so that both double sums converge.
Then
\begin{align}
&A_{d , b}(u,v|4\tau)-C_{d , b}(u,v|4\tau)
=\vartheta_4(u|(2d+b)\tau)\vartheta_4(v|(2d-b)\tau).\label{c_gen}
\end{align}
Moreover,
\begin{align}\label{ACsqeqn}
&A_{d, b}^2(u,v|\tau)-C_{d,b}^2(u,v|\tau) = A_{d, b}^2(u,v+\pi/2|\tau)-C_{d,b}^2(u,v+\pi/2|\tau).
\end{align}
\end{theorem}

\begin{proof}
Every integer pair $(m,n)\in \mathbf{Z}^2$ corresponds uniquely to a  pair $(k,\ell) \in\mathbf{Z}^2$ satisfying $k\equiv \ell \pmod{2}$ via the transformation
\[
m=\frac{k+\ell}{2} \quad \text{and} \quad  n=\frac{k-\ell}{2}.
\]
Therefore, we may express a double sum over $(m,n)\in\mathbf{Z}^2$ as
\begin{align*}
\sum_{m,n \in \mathbf{Z}}F(m,n) &= \sum_{\substack{k,\ell\in\mathbf{Z}\\ k\equiv \ell \pmod{2}}}F\left(\frac{k+\ell}{2}, \frac{k-\ell}{2}\right) \\
&= \sum_{k,\ell\in\mathbf{Z}}F\left(k+\ell, k-\ell\right)+\sum_{k,\ell\in\mathbf{Z}}F\left(k+\ell+1, k-\ell\right),
\end{align*}
by considering $k,\ell$ both even and $k,\ell$ both odd.

Applying this change of summation indices on $A_{d, b}^2(u,v|\tau)$ and
$C_{d, b}^2(u,v|\tau)$, respectively, we find that
\begin{align} \nonumber
A_{d, b}(u,v|\tau)
&= \vartheta_3(2u|(2d+b)\tau)\vartheta_3(2v|(2d-b)\tau)\nonumber\\
&\quad+\vartheta_2(2u|(2d+b)\tau)\vartheta_2(2v|(2d-b)\tau)\label{Ceqn100} \intertext{and}
C_{d,b}(u,v|\tau)
&= \vartheta_2(2u|(2d+b)\tau)\vartheta_3(2v|(2d-b)\tau)\notag \\ \label{Ceqn}
&\quad+\vartheta_3(2u|(2d+b)\tau)\vartheta_2(2v|(2d-b)\tau).
\end{align}
Taking the difference  between \eqref{Ceqn100} and \eqref{Ceqn}, we find that
\begin{align*}
A_{d, b}(u,v|4\tau) - C_{d, b}(u,v|4\tau)
&=\left[\vartheta_3(2u|4(2d+b)\tau)-\vartheta_2(2u|4(2d+b)\tau)\right]\\
&\quad\times
\left[\vartheta_3(2v|4(2d-b)\tau)
-\vartheta_2(2v|4(2d-b)\tau)\right]\\
&=\vartheta_4(u|(2d+b)\tau)\vartheta_4(v|(2d-b)\tau),
\end{align*}
since
\[
\vartheta_3(2z|4\tau)-\vartheta_2(2z|4\tau)= \vartheta_4(z|\tau),
\]
which follows directly from the definitions of the theta functions.
This proves \eqref{c_gen}.

Next, we prove \eqref{ACsqeqn}.   It is clear from the
definitions of $A_{d,b}(u,v)$ and $C_{d,b}(u,v)$, \eqref{A2def} and \eqref{C2def}, that
\begin{align} \label{AdbPi}
A_{d, b}(u+\pi/2,v+\pi/2|\tau) &= A_{d, b}(u,v|\tau)
\intertext{and} \label{CdbPi}
C_{d, b}(u+\pi/2,v+\pi/2|\tau) &= -C_{d, b}(u,v|\tau).
\end{align}
Replacing $u$ and $v$ by $u+\pi/2$ and $v+\pi/2$ \eqref{c_gen} respectively, followed by using \eqref{AdbPi} and \eqref{CdbPi}, we obtain
\[
A_{d, b}(u,v|4\tau) + C_{d, b} (u,v|4\tau)
=\vartheta_4(u+\pi/2|(2d+b)\tau)\vartheta_4(v+\pi/2|(2d-b)\tau).
\]
Multiplying this equation by \eqref{c_gen}, we then obtain
\begin{align}
A_{d, b}^2(u,v|4\tau) - C_{d, b}^2 (u,v|4\tau)
&= \vartheta_4(u|(2d+b)\tau)\vartheta_4(v|(2d-b)\tau)\notag\\
& \qquad \times \vartheta_4(u+\pi/2|(2d+b)\tau)\vartheta_4(v+\pi/2|(2d-b)\tau).
\label{Ceqn101}
\end{align}
Since the expression on the right hand side
of \eqref{Ceqn101}
is invariant when $v$ is replaced by $v+\pi/2$,  we obtain
$$A_{d, b}^2(u,v|4\tau) - C_{d, b}^2(u,v|4\tau)=
A_{d, b}^2(u,v+\pi/2|4\tau) - C_{d, b}^2(u,v+\pi/2|4\tau)$$
and the proof of Lemma \ref{AC2thm} is complete.
\end{proof}

\begin{corollary}[Generalization of \eqref{ChanChanSole-id}] \label{aba} Suppose $b$ and $d$ are integers satisfying
$-2d < b< 2d$. 
Then
\begin{align} \nonumber
&\left(\sum_{m,n\in\mathbf{Z}} (-1)^{n-m} q^{d(m+1/2)^2+b(m+1/2)(n+1/2)+d(n+1/2)^2} x^{m+1/2} y^{n+1/2}
\right)^2\\ \nonumber
&+
\left(\sum_{m,n\in\mathbf{Z}} q^{dm^2+b mn+dn^2} x^{m} y^{n}
\right)^2=
\left(\sum_{m,n\in\mathbf{Z}} (-1)^{n-m} q^{d m^2+b mn+d n^2} x^{m} y^{n}
\right)^2\\ & +
\left(\sum_{m,n\in\mathbf{Z}} q^{d(m+1/2)^2+b(m+1/2)(n+1/2)+d(n+1/2)^2} x^{m+1/2} y^{n+1/2}
\right)^2.\label{gen-Jacobi-xy}
\end{align}
\end{corollary}

\begin{proof}
Identity \eqref{gen-Jacobi-xy} follows by letting $$x=e^{4iu}\,\,\text{and}\,\, y=e^{2i(u+v)}$$ in \eqref{ACsqeqn}.
\end{proof}

Note that Theorem~\ref{th-Jacobi} follows from Corollary \ref{aba} when $d=1$ and $b=0$.

\begin{theorem}\label{AC2thm2}
Let
\begin{align} \label{A2def2}
\widehat{A}_{\alpha , \beta}(u,v|\tau)&=\sum_{m,n\in\mathbf{Z}} q^{\alpha m^2+\alpha mn+\beta n^2}e^{2iu(2m+n)} e^{2ivn}
\intertext{and} \label{C2def2}
\widehat{C}_{\alpha , \beta}(u,v|\tau)
&=\sum_{m,n\in\mathbf{Z}} q^{\alpha m^2+ \alpha mn+\beta n^2+\alpha m+\alpha n/2+\alpha /4}
e^{2iu(2m+n+1)} e^{2ivn}.
\end{align}
Let $0<\alpha<4\beta$ so that both double sums converge.
Then
\begin{align}
&\widehat{A}_{\alpha , \beta}(u,v|4\tau)-\widehat{C}_{\alpha , \beta}(u,v|4\tau)
=\vartheta_4(u|\alpha\tau)\vartheta_4(v|(4\beta-\alpha)\tau).\label{c_gen2}
\end{align}
Moreover,
\begin{align}\label{ACsqeqn2}
&\widehat{A}_{\alpha, \beta}^2(u,v|\tau)-\widehat{C}_{\alpha,\beta}^2(u,v|\tau) = \widehat{A}_{\alpha, \beta}^2(u,v+\pi/2|\tau)-\widehat{C}_{\alpha,\beta}^2(u,v+\pi/2|\tau).
\end{align}
\end{theorem}

\begin{proof}
Instead of proving \eqref{c_gen2} directly, thereby replicating similar proofs to that of \eqref{c_gen}, we deduce \eqref{c_gen2} from \eqref{c_gen}.
By replacing the summation variables $(m,n)$ by $(\ell, k-\ell)$, we find that
\begin{align*}
\widehat{A}_{\alpha , \beta}(u,v|\tau) &= A_{\beta , \alpha-2\beta}(u,v|\tau)
\intertext{and}
\widehat{C}_{\alpha , \beta}(u,v|\tau) &= C_{\beta , \alpha-2\beta}(u,v|\tau).
\end{align*}
Therefore, by applying \eqref{c_gen},
\begin{align*}
&\widehat{A}_{\alpha , \beta}(u,v|4\tau)-\widehat{C}_{\alpha , \beta}(u,v|4\tau)\\
&=A_{\beta , \alpha-2\beta}(u,v|4\tau)- C_{\beta , \alpha-2\beta}(u,v|4\tau)\\
&=\vartheta_4(u|\alpha\tau)\vartheta_4(v|(4\beta-\alpha)\tau),
\end{align*}
which is \eqref{c_gen2}. Identity \eqref{ACsqeqn2} also follows easily from \eqref{ACsqeqn} by noting that
\begin{align*}
\widehat{A}_{\alpha, \beta}^2(u,v|\tau)-\widehat{C}_{\alpha,\beta}^2(u,v|\tau)
&= A_{\beta, \alpha-2\beta}^2(u,v|\tau)-C_{\beta,\alpha-2\beta}^2(u,v|\tau) \\
&= A_{\beta, \alpha-2\beta}^2(u,v+\pi/2|\tau)-C_{\beta,\alpha-2\beta}^2(u,v+\pi/2|\tau)\\
&=\widehat{A}_{\alpha, \beta}^2(u,v+\pi/2|\tau)-\widehat{C}_{\alpha,\beta}^2(u,v+\pi/2|\tau).
\qedhere
\end{align*}
\end{proof}

\begin{corollary}[Generalization of \eqref{CCS-gen}] \label{aab-1}
Suppose $\alpha$ and $\beta$ are integers satisfying $0<\alpha<4\beta$. Then
\begin{align} \nonumber
&\left(\sum_{m,n\in\mathbf{Z}} (-1)^{n} q^{\alpha (m+1/2)^2+\alpha (m+1/2)n+\beta n^2} x^{m+1/2} y^{n}
\right)^2 \\ \nonumber
&+
\left(\sum_{m,n\in\mathbf{Z}} q^{\alpha m^2+\alpha mn+\beta n^2} x^{m} y^{n}
\right)^2
=
\left(\sum_{m,n\in\mathbf{Z}} (-1)^{n} q^{\alpha m^2+\alpha mn+\beta n^2} x^{m} y^{n}
\right)^2\\ 
&\qquad +
\left(\sum_{m,n\in\mathbf{Z}} q^{\alpha (m+1/2)^2+\alpha (m+1/2)n+\beta n^2} x^{m+1/2} y^{n}
\right)^2.\label{jacobi-abc2}
\end{align}
\end{corollary}

\begin{proof}
Identity \eqref{jacobi-abc2} follows by letting $$x=e^{4iu}\,\,\text{and}\,\, y=e^{2i(u+v)}$$ in \eqref{ACsqeqn2}.
\end{proof}

\begin{remark}
The identity \eqref{ACsqeqn} is a two-variable generalization of \eqref{ChanChanSole-id} and can be shown as follow:
From \eqref{Ceqn}, we find that
\begin{equation} C_{d,b}(0,0|\tau)-C_{d,b}(0,\pi/2|\tau)
= 2\vartheta_3(0|(2d+b)\tau)\vartheta_2(0|(2d-b)\tau)\label{C2minus}
\end{equation}
and
\begin{equation}\label{C2plus}
C_{d,b}(0,0|\tau)+C_{d,b}(0,\pi/2|\tau)
= 2\vartheta_2(0|(2d+b)\tau)\vartheta_3(0|(2d-b)\tau).
\end{equation}
Multiplying \eqref{C2minus} by \eqref{C2plus}, we arrive at
\begin{align*}
&C_{d,b}^2(0,0|\tau)-C_{d,b}^2(0,\pi/2|\tau)\\
&=4\vartheta_2(0|(2d+b)\tau)\vartheta_3(0|(2d+b)\tau)
\vartheta_2(0|(2d-b)\tau)\vartheta_3(0|(2d-b)\tau)\\
&=\vartheta_2^2(0|(d+b/2)\vartheta_2^2(0|(d-b/2)\tau),
\end{align*}
where we have used  \cite[p.58]{Chan-book}
\begin{equation} \label{t2322}
2\vartheta_2(0|2\tau)\vartheta_{3}(0|2\tau)=\vartheta_2^2(0|\tau).
\end{equation}

Applying the change of summation indices $m=(k+\ell)/2$ and $n=(k-\ell)/2$,  as in the proof of  Theorem \ref{AC2thm}, we find that
\begin{align*}
&2\sum_{m,n\in\mathbf{Z}} q^{2(d(m+1/2)^2+b(m+1/2)n+d n^2)}\\
&=2\sum_{k, \ell \in\mathbf{Z}} q^{2(2d+b)(k+1/4)^2+(2d-b)(\ell+1/4)^2}
+2\sum_{k, \ell \in\mathbf{Z}} q^{2(2d+b)(k+3/4)^2+(2d-b)(\ell+3/4)^2}\\
&=4\sum_{k, \ell \in\mathbf{Z}} q^{2(2d+b)(k+1/4)^2+(2d-b)(\ell+1/4)^2}\\
&=\vartheta_2(0|(d+b/2)\vartheta_2(0|(d-b/2)\tau),
\end{align*}
where the last equality follows from the following computations:
\begin{align*}
&2\sum_{k\in\mathbf{Z}} q^{4(k+1/4)^2}
= 2\sum_{k\in\mathbf{Z}} q^{4(k-1/4)^2}
= 2\sum_{k\in\mathbf{Z}} q^{4(k+3/4)^2}\\
&=\sum_{k\in\mathbf{Z}} q^{(4k+1)^2/4} + \sum_{k\in\mathbf{Z}} q^{(4k+3)^2/4}
=\sum_{k\in\mathbf{Z}} q^{(2k+1)^2/4} = \vartheta_2(0|\tau).
\end{align*}
Therefore,
\begin{align*}
&C_{d,b}^2(0,0|\tau)-C_{d,b}^2(0,\pi/2|\tau)
=\left(2\sum_{m,n\in\mathbf{Z}} q^{2(d(m+1/2)^2+b(m+1/2)n+d n^2)} \right)^2,
\end{align*}
and by setting $u=v=0$ in \eqref{ACsqeqn}, we recover \eqref{ChanChanSole-id}.
\end{remark}

\begin{remark}
Identity \eqref{ACsqeqn2} is a two variable generalization of \cite[(1.6)]{Chan-Chan-Sole} and we present the proof as follow:
Comparing \eqref{ACsqeqn2} and \cite[(1.6)]{Chan-Chan-Sole}, we see that it suffices to show
\begin{align}
&\widehat{A}_{\alpha, \beta}^2(0,\pi/2|2\tau) - \widehat{C}_{\alpha, \beta}^2(0,\pi/2|2\tau)= \left(\sum_{m,n\in\mathbf{Z}} (-1)^{m-n}q^{\alpha m^2+ \alpha mn +\beta n^2}\right)^2.
\label{AC=B}
\end{align}

From \eqref{Ceqn101}, we see that
\begin{align*}
&\widehat{A}_{\alpha, \beta}^2(0,\pi/2|4\tau) - \widehat{C}_{\alpha, \beta}^2(0,\pi/2|4\tau)\\
&= \vartheta_4(0|\alpha\tau)\vartheta_4(0|(4\beta-\alpha)\tau)
 \vartheta_4(\pi/2|\alpha\tau)\vartheta_4(\pi/2|(4\beta-\alpha)\tau)\\
 &= \vartheta_4(0|\alpha\tau)\vartheta_4(0|(4\beta-\alpha)\tau)
 \vartheta_3(0|\alpha\tau)\vartheta_3(0|(4\beta-\alpha)\tau)\\
 &= \vartheta_4^2(0|2\alpha\tau)\vartheta_4^2(0|2(4\beta-\alpha)\tau)
\end{align*}
by applying the   identity \cite[p.34]{BB}
\begin{equation} \label{t3444}
\vartheta_3(0|\tau)\vartheta_4(0|\tau) = \vartheta_4^2(0|2\tau).
\end{equation}

Since it is shown in \cite[(3.2)]{Chan-Chan-Sole} that
\begin{equation*} \nonumber
 \sum_{m,n\in\mathbf{Z}} (-1)^{m-n}q^{\alpha m^2+ \alpha mn +\beta n^2}
 = \vartheta_4(0|\alpha\tau)\vartheta_4(0|(4\beta-\alpha)\tau),
\end{equation*}
equation \eqref{AC=B} then follows immediately.
\end{remark}

\begin{remark}
Both \eqref{ChanChanSole-id} and \eqref{CCS-gen} are generalizations of \eqref{Jacobi}. This can be seen as follows.

By setting $b=1$ and $d=0$ in \eqref{ChanChanSole-id}, the equivalence between \eqref{ChanChanSole-id} and
\eqref{Jacobi} is established once we show that
\begin{equation}\label{t2322a}
2\sum_{m,n=-\infty}^{\infty} q^{2(m+1/2)^2+2n^2} =
\sum_{m,n=-\infty}^{\infty} q^{(m+1/2)^2+(n+1/2)^2},
\end{equation}
but this is precisely identity \eqref{t2322}.

Next, to show \eqref{CCS-gen} implies \eqref{Jacobi}, we set $b=1$ and $d=1/2$ in \eqref{CCS-gen}, and then replace the summation indices $(m,n)$  by $(k,\ell-k)$. Doing so, we arrive at
\begin{align*}&\left(\sum_{m,n=-\infty}^\infty (-1)^{n} q^{(m^2+n^2)/2}\right)^2
+\left(\sum_{m,n=-\infty}^\infty q^{(m+1/2)^2+(n+1/2)^2}\right)^2\notag\\
&\qquad\qquad\qquad\qquad=\left(\sum_{m,n=-\infty}^\infty q^{m^2+n^2}\right)^2.
\end{align*}
Thus, it suffices to show that
\begin{equation}\label{t3444a}
\sum_{m,n=-\infty}^\infty (-1)^{n} q^{m^2+n^2}
= \sum_{m,n=-\infty}^\infty (-1)^{m+n} q^{2(m^2+n^2)},
\end{equation}
which again is precisely identity \eqref{t3444}.
\end{remark}

\section{Cubic Borweins-type identities} \label{new-borweins}
\label{s6}

It is natural to ask if there are cubic analogues of \eqref{ChanChanSole-id} and \eqref{CCS-gen}.
This question leads us to the following new identities of Borweins type.

\begin{theorem} \label{SH-cubic}
Let $|q|<1$ and $\omega=e^{2\pi i /3}$. Then
\begin{align}
\left(\sum_{m,n\in \mathbf Z} \omega^{m-n}q^{m^2+mn+n^2}\right)^3
&+\left(3\sum_{m,n\in\mathbf{Z}} q^{3((m+1/3)^2+(m+1/3)n+n^2)}\right)^3\notag\\
 &\qquad\qquad= \left(\sum_{m,n\in \mathbf Z} q^{m^2+mn+n^2}\right)^3 \label{ChanChanSole-(1,1)}
 \intertext{and}
\left(\sum_{m,n\in \mathbf Z} \omega^{m-n}q^{3m^2+3mn+n^2}\right)^3
&+\left(\sum_{m,n\in\mathbf{Z}} q^{3(3(m+1/3)^2+3(m+1/3)n+n^2)}\right)^3\notag\\
 &\qquad\qquad= \left(\sum_{m,n\in \mathbf Z} q^{3(3m^2+3mn+n^2)}\right)^3. \label{CCS-gen-(3,1)}
\end{align}
 \end{theorem}

We note that \eqref{ChanChanSole-(1,1)} is the cubic analogue of \eqref{ChanChanSole-id} in the case $b=d=1$ and
\eqref{CCS-gen-(3,1)} is the cubic analogue of \eqref{CCS-gen} in the case $b=3$ and $d=1$. These identities are
different ways of expressing \eqref{Borweins} but establishing their connections with \eqref{Borweins} turns out to be very interesting on its own.

We first require three identities.

\begin{lemma} Let $|q|<1$ and $\omega=e^{2\pi i/3}$. Then
\begin{align}
&\sum_{m,n\in\mathbf Z} \omega^{m-n} q^{m^2+mn+n^2+m+n} = 0,\label{B-0}\\
&\sum_{m,n\in\mathbf{Z}} \omega^{m+n}q^{m^2+mn+n^2}
=\sum_{m,n\in\mathbf{Z}} \omega^{m-n}q^{3(m^2+mn+n^2)} \label{conj1}
\intertext{and} \nonumber
&\sum_{m,n\in\mathbf{Z}} q^{(m+1/3)^2+(m+1/3)(n+1/3)+(n+1/3)^2}\\
&=3\sum_{m,n\in\mathbf{Z}} q^{3((m+1/3)^2-(m+1/3)(n+1/3)+(n+1/3)^2)}. \label{conj2}
\end{align}
\end{lemma}

\begin{proof}
Identity \eqref{B-0} is \eqref{Borweins-0} and the proof is presented in Section~\ref{s1}.
Let
\begin{gather*}
N(m+n\gamma) = \left|m+n\gamma\right|^2= m^2+mn+n^2,\\
R(m+n\gamma)= m \quad\text{and}\quad
I(m+n\gamma)= n,
\end{gather*}
where $\gamma=(1+\sqrt{-3})/2.$
From
$$\mathcal L=\bigcup_{j=0}^2 \mathcal L_j$$
where $$\mathcal L= \{m+n\gamma\mid m,n\in\mathbf Z\}\quad\text{and}\quad \mathcal L_j = \{-2n-m+j+(2m+n)\gamma\mid m,n\in\mathbf Z\},$$
we deduce that
$$
\sum_{\ell\in \mathcal L} \omega^{R(\ell)+I(\ell)} q^{N(\ell)}
=\sum_{j=0}^2 T_j$$
where $$T_j =\sum_{\ell\in \mathcal L_j} \omega^{R(\ell)+I(\ell)} q^{N(\ell)}.$$
Note that
\begin{align*}
T_0 &= \sum_{m,n\in \mathbf Z}\omega^{m-n} q^{3m^2+3mn+3n^2},\\
T_1 &= \sum_{m,n\in \mathbf Z}\omega^{m-n+1} q^{3m^2+3mn+3n^2-3n+1}\intertext{and}
T_2 &= \sum_{m,n\in \mathbf Z}\omega^{m-n-1} q^{3m^2+3mn+3n^2+3n+1}.
\end{align*}
Replacing $n$ by $n+1$ in $T_1$,
we deduce that
$$T_1=\sum_{m,n\in \mathbf Z}\omega^{m-n} q^{3m^2+3mn+3n^2+3m+3n+1}=0,$$
by \eqref{B-0}.
Similarly, by replacing $n$ by $n-1$ in $T_2$, followed by $(m,n)$ by $(-n,-m)$, and using \eqref{B-0}, we conclude that $T_2=0$ and the proof of \eqref{conj1} is complete.

To prove \eqref{conj2}, we observe that by direct verification that
$$\sum_{m,n\in \mathbf Z} q^{3m^2+3n^2+3mn+3m+3n+1} = \sum_{\substack{m,n\in \mathbf Z\\ m\equiv n+1 \pmod{3}}} q^{m^2+mn+n^2}.$$
The solutions to $m\equiv n+1\pmod{3}$ are $(m,n)\equiv (1,0), (-1,1) $ and $(0,-1) \pmod{3}$. Therefore,
$$\sum_{\substack{m,n\in \mathbf Z\\ m\equiv n+1\pmod{3}}} q^{m^2+mn+n^2}=
C_{1,0}+C_{0,-1}+C_{-1,1},$$
where
\begin{equation*}
C_{i,j} = \sum_{\substack{m,n\in \mathbf Z\\ (m,n)\equiv (i,j) \pmod{3}}} q^{m^2+mn+n^2}.
\end{equation*}
We will show that $C_{1,0}=C_{0,-1}=C_{-1,1}$.
Note that
$$C_{1,0} = C_{-1,0}=C_{0,-1}.$$

To prove that $C_{-1,1}=C_{1,0}$, we observe that
\begin{equation}\label{C(-1,1)} C_{-1,1}=\sum_{k,\ell\in \mathbf Z} q^{9k^2+9k\ell+9\ell^2-3k+3\ell+1}\end{equation}
and
$$C_{1,0}=\sum_{k,\ell\in \mathbf Z} q^{9k^2+9k\ell+9\ell^2+6k+3\ell+1}$$ and that
$C_{-1,1}=C_{1,0}$ after replacing $(k,\ell)$ by $(-k-\ell,k)$ in $C_{-1,1}.$
Therefore,
\begin{equation}\label{c-C(-1,1)}\sum_{m,n\in \mathbf Z} q^{3(m+1/3)^2+3(m+1/3)(n+1/3)+3(n+1/3)^2}=3C_{-1,1}
\end{equation} and the proof of \eqref{conj2} is complete.
\end{proof}

\begin{remark}\quad
\smallskip

\begin{enumerate}[1.]
\item
Identity \eqref{B-0} can also be proved by writing $S$
as
$$S=\omega^{-1}S_{-1}+S_0+\omega S_1,$$
where $$S_j = \sum_{\substack{m,n\in \mathbf Z\\ m-n\equiv j \pmod{3}}} q^{m^2+mn+n^2+m+n}.$$ It can be shown that
that $S_{-1}=S_0=S_{1}$, so that
$$S=(\omega^{-1}+1+\omega) S_0 = 0$$ since $1+\omega+\omega^{-1}=0$.
We leave the details to the reader.

\medskip
\item Identity \eqref{conj2} is not new. The identity and its proof can be found in \cite{Ye}, \cite[Lemma 3.13]{Cooper} and \cite[(1.11)]{Chan-Wang}.

\medskip
\item
Identities \eqref{conj1} and \eqref{conj2} (in the equivalent form \eqref{conj2-1} below) can be viewed as
  cubic analogues of \eqref{t3444a} and \eqref{t2322a}, respectively.
\end{enumerate}
\end{remark}

We are now ready to prove Theorem \ref{SH-cubic}.
\begin{proof}[Proof of Theorem \ref{SH-cubic}]
By replacing $(m,n)$ by  $(m,m+n)$ for the series on the right hand side of \eqref{conj2}, we deduce that
\begin{equation} \label{conj2-1}
\sum_{m,n\in\mathbf{Z}}  q^{(m+1/3)^2+(m+1/3)(n+1/3)+(n+1/3)^2}
=3\sum_{m,n\in\mathbf{Z}} q^{3((m+1/3)^2+(m+1/3)n+n^2)}.
\end{equation}
Thus Borweins' cubic identity \eqref{Borweins} can be expressed as
\begin{align*}
\left(\sum_{m,n\in \mathbf Z} \omega^{m-n}q^{m^2+mn+n^2}\right)^3
&+\left(3\sum_{m,n\in\mathbf{Z}} q^{3((m+1/3)^2+(m+1/3)n+n^2)}\right)^3\notag\\
 &\qquad\qquad= \left(\sum_{m,n\in \mathbf Z} q^{m^2+mn+n^2}\right)^3,
\end{align*}
which is the cubic analog of \cite[eq. (1.5)]{Chan-Chan-Sole} with $b=d=1$.

Replacing $(m,n)$ by $(m,m+n)$ in the second series on the left hand side  and the series on the right hand side of \eqref{Borweins},  we find that
\begin{align*}
\left(\sum_{m,n\in \mathbf Z} \omega^{m-n}q^{m^2+mn+n^2}\right)^3
&+\left(\sum_{m,n\in\mathbf{Z}} q^{3(m+1/3)^2+3(m+1/3)n+n^2}\right)^3\notag\\
 &\qquad\qquad= \left(\sum_{m,n\in \mathbf Z} q^{3m^2+3mn+n^2}\right)^3, 
\end{align*}
which, after replacing $q$ by $q^3$, that
\begin{align}
\left(\sum_{m,n\in \mathbf Z} \omega^{m-n}q^{3(m^2+mn+n^2)}\right)^3
&+\left(\sum_{m,n\in\mathbf{Z}} q^{3(3(m+1/3)^2+3(m+1/3)n+n^2)}\right)^3\notag\\
 &\qquad\qquad= \left(\sum_{m,n\in \mathbf Z} q^{3(3m^2+3mn+n^2)}\right)^3. \label{Borweins-1}
\end{align}
But by \eqref{conj1},
\begin{align}
\sum_{m,n\in\mathbf{Z}} \omega^{m-n}q^{3(m^2+mn+n^2)}
&=
\sum_{m,n\in\mathbf{Z}} \omega^{m+n}q^{m^2+mn+n^2}\notag \\
&=\sum_{m,n\in\mathbf{Z}} \omega^{2m+n}q^{3m^2+3mn+n^2}\notag\\
&=\sum_{m,n\in\mathbf{Z}} \omega^{n-m}q^{3m^2+3mn+n^2},\label{CCS-1}
\end{align}
where the second equality follows by replacing $(m,n)$ by $(m,m+n)$ and the third equality follows from $\omega^{2m}=\omega^{-m}.$
Substituting \eqref{CCS-1} into \eqref{Borweins-1},
we arrive at \eqref{CCS-gen-(3,1)}.
\end{proof}

\begin{remark}%
Using \eqref{C(-1,1)}, \eqref{c-C(-1,1)}, \eqref{Borweins} and the fact that
$$C_{-1,1}= \sum_{m,n\in\mathbf{Z}} q^{9((m+1/3)^2-(m+1/3)(n+1/3)+(n+1/3)^2)},$$
we deduce that
\begin{align}
\left(\sum_{m,n\in \mathbf Z} \omega^{m-n}q^{m^2+mn+n^2}\right)^3
&+\left(3\sum_{m,n\in\mathbf{Z}} q^{3((m+1/3)^2-(m+1/3)(n+1/3)+(n+1/3)^2)}\right)^3\notag\\
 &\qquad\qquad= \left(\sum_{m,n\in \mathbf Z} q^{m^2+mn+n^2}\right)^3 \label{Borweins-alt-1}
\end{align}
Using \eqref{conj1} and \eqref{Borweins}, we arrive at
\begin{align}
\left(\sum_{m,n\in \mathbf Z} \omega^{m+n}q^{m^2+mn+n^2}\right)^3
&+\left(\sum_{m,n\in\mathbf{Z}} q^{3((m+1/3)^2+(m+1/3)(n+1/3)+(n+1/3)^2)}\right)^3\notag\\
 &\qquad\qquad= \left(\sum_{m,n\in \mathbf Z} q^{3(m^2+mn+n^2)}\right)^3. \label{Borweins-alt-2}
\end{align}
Identities \eqref{Borweins-alt-1} and \eqref{Borweins-alt-2} are equivalent to \eqref{Borweins-xy} and their discoveries are motivated by
\cite{Chan-Chan-Sole}.
\end{remark}

\section{Concluding remarks}
\label{s7}

In this article, we have shown that $A(x,y|\tau)$ and $C(x,y|\tau)$ are coefficients of certain theta series identities such as
\eqref{cubthetaen3}. We then use \eqref{cubthetaen3} to prove \eqref{Borweins-xy}. We have also connected \eqref{Borweins-xy} to identity connected to the study
of Macdonald identities. Using transformation formulas from classical texts, we derive new analogues of \eqref{Borweins-xy}.
The function $A(x,y|\tau)$ and its variations have appeared before the work of Schultz. For more details, see \cite{Bhargava}, \cite{Hirschhorn-Garvan-Borwein}, \cite{Chapman} and \cite{Cooper-article}.

In Section 5, we discover infinitely many theta series
$X,Y,Z,W$ associated with the binary quadratic forms $bm^2+bmn+dn^2$ and $dm^2+bmn+dn^2$ with $|b|<2d$, that fit the relation
$$X^2+Y^2=Z^2+W^2.$$
We are, however, unable to find theta series $X, Y, Z, W$ that satisfy
\begin{equation}\label{XYZW-3} X^3+Y^3=Z^3+W^3\end{equation}
besides those that we have presented here. In particular, we are not able to find theta series satisfying \eqref{XYZW-3} which are not associated with
the binary quadratic form $m^2+mn+n^2.$

In Section~\ref{s6}, we derive new series of one-variable $A, B, C$ that satisfy
\begin{equation}\label{ABC-3} A^3+B^3=C^3.\end{equation}
Although we have infinitely many triplets of series of one-variable $A, B, C$ satisfying $$A^2+B^2=C^2,$$
\eqref{ChanChanSole-(1,1)}
and \eqref{CCS-gen-(3,1)} are the only solutions we found for \eqref{ABC-3} and
these identities are different ways of expressing \eqref{Borweins}. We do not know if there are theta series of one-variable that satisfy \eqref{ABC-3} which are not associated with the binary quadratic form $m^2+mn+n^2$.

\medskip
\noindent
\textbf{Acknowledgments.}
We thank the anonymous referees for their healthy criticism that helped to improve the original version of this manuscript.
We also thank Siew Lian Tan for her help in improving the presentation of the article.

\end{document}